\documentclass[12pt]{amsart}
\usepackage{amssymb,amsmath,amsthm}
\numberwithin{equation}{section}
\begin{document}

\theoremstyle{plain}
\newtheorem{theorem}{Theorem}[section]
\newtheorem{lemma}[theorem]{Lemma}
\newtheorem{proposition}[theorem]{Proposition}
\newtheorem{corollary}[theorem]{Corollary}
\newtheorem{conjecture}[theorem]{Conjecture}

\theoremstyle{definition}
\newtheorem*{definition}{Definition}

\theoremstyle{remark}
\newtheorem*{remark}{Remark}
\newtheorem{example}{Example}[section]
\newtheorem*{remarks}{Remarks}

\newcommand{\cc}{{\mathbb C}}
\newcommand{\qq}{{\mathbb Q}}
\newcommand{\rr}{{\mathbb R}}
\newcommand{\nn}{{\mathbb N}}
\newcommand{\zz}{{\mathbb Z}}
\newcommand{\pp}{{\mathbb P}}
\newcommand{\al}{\alpha}
\newcommand{\be}{\beta}
\newcommand{\ga}{\gamma}
\newcommand{\ze}{\zeta}
\newcommand{\om}{\omega}
\newcommand{\ep}{\epsilon}
\newcommand{\la}{\lambda}
\newcommand{\de}{\delta}
\newcommand{\De}{\Delta}
\newcommand{\si}{\sigma}
\newcommand{\ta}{\theta}
 \newcommand{\Exp}{{\rm Exp}}
\newcommand{\legen}[2]{\genfrac{(}{)}{}{}{#1}{#2}}
\def\End{{\rm End}}
\newcommand{\uu}{{\mathcal U}}
\newcommand{\tu}{\tilde u}

\newcommand{\ts}{\thinspace}
\newcommand{\sse}{\subseteq}

\newcommand{\rx}{\rr[x]}
\newcommand{\new}{\text{New}}
\newcommand{\ux}{\underline{x}}
\newcommand{\vsp}{\vspace{.1cm}}

% Sept. 17, 2019

\title{A note on mediated simplices}

\begin{abstract} Many homogeneous polynomials that arise in the study of sums of squares and 
Hilbert's 17th problem come from  monomial substitutions into the arithmetic-geometric inequality.  
In 1989, the second author gave a
necessary and sufficient condition for such a form to have a representation as a sum of squares of 
forms [\emph{Math.~Ann.~}\textbf{283}, 431--464] , 
involving the arrangement of  lattice points in the simplex whose vertices were the $n$-tuples 
of the exponents used in the substitution. Further, a claim was made, and not proven, that 
sufficiently large dilations of any 
such simplex will also satisfy
 this condition. The aim of this short note is to prove the claim, and provide further context for
 the result, both in the study of Hilbert's 17th Problem and the study of lattice point simplices. 
 \end{abstract}

\author[V.~Powers]{Victoria Powers}
\author[B.~Reznick]{Bruce Reznick}
\address[V.~Powers]{Department of Mathematics, Emory University,
Atlanta, GA 30322}
\email{vpowers@emory.edu}
\address[B.~Reznick]{Department of Mathematics, University of 
Illinois at Urbana-Champaign, Urbana, IL 61801} 
\email{reznick@illinois.edu}
\date{\today} 

\maketitle

\section{Introduction}

In 1989, the second author considered  \cite{Re3} a class of homogeneous polynomials 
(forms) 
which had arisen in the study of Hilbert's 17th Problem as monomial substitutions into the 
arithmetic-geometric inequality. The goal was to determine when such a form, which must be 
positive semidefinite, had a representation  as a sum of squares of forms. The answer was a 
necessary and sufficient condition 
involving the arrangement of  lattice points in the simplex whose vertices were the $n$-tuples 
of the exponents used in the substitution. 
Further, a claim was made in  \cite{Re3}, and not proven, that sufficiently large dilations of any 
such simplex will also satisfy
 this condition. The aim of this short note is to prove the claim, and provide further context for
 the result, both in the study of Hilbert's 17th Problem and the study of lattice point simplices. 
 The second author is happy to acknowledge that the return to this claim was triggered by two
 nearly simultaneous events: an invitation to speak at the 2019 SIAM Conference on Applied
 Algebraic Geometry, and a request from Jie Wang for a copy of \cite{Re4}, which was
announced in \cite{Re3} but  never written.

 \section{Preliminaries}

We work with homogeneous polynomials (forms) in 
$\rx = \rr[x_1, \dots , x_n]$, the ring of real polynomials in $n$ variables.  
Write the monomial $x_1^{\al_1}\cdots x_n^{\al_n}$ as $x^{\al}$, for 
$\al = (\al_1,\dots,\al_n) \in \zz^n$.  For $p(x)  = \sum_{\al} c(\al)x^{\al} \in \rx$, let 
supp$(p) = \{ \al \mid c(\al) \neq 0 \}$, write
$\new(p)$ for the {\it Newton polytope} of $p$, that is, the convex hull of supp$(p)$, and 
let $C(p) = \new(p) \cap \zz^n$.

A form $p \in \rx$ is {\it positive semidefinite} or {\it psd} if $p(x) \geq 0$ for all $x \in \rr^n$.  
It is a {\it sum of squares} or {\it sos} if $p = \sum_j h_j^2$
for forms $h_j \in \rx$.  Clearly, every sos form is psd.  In 1888, D. Hilbert \cite{Hil1} 
proved that there exist psd forms which are not sos. 

The {\it arithmetic-geometric inequality} (or {\it AGI})  states that if $t_i \ge 0$, 
$\la_i \ge 0$ and $\sum_{i=1}^n \la _i = 1$, then
\[
\la_1 t_1 + \cdots + \la_n t_n \ge t_1^{\la_1} \cdots t_n^{\la_n},
\]
with equality only if the $t_i$'s are equal. In 1891, A. Hurwitz \cite{Hu}  gave a proof of the AGI,
in which the key step was setting $\la_i  = a_i/N$ where $a_i \in \zz^n$ with  $\sum a_i = N$ for 
even $N$, and  $t_i = x_i^N$. Under this substitution and a scaling,  one obtains the form
\[
a_1 x_1^N + \cdots + a_n x_n^N  - N x_1^{a_1} \cdots x_n^{a_n}.
\]
Hurwitz then proves that each such form is sos (in fact, a sum of squares of
binomials), and hence psd. (He cites \cite{Hil1} to observe that this is not automatic.)
For example, after a scaling and relabeling of the $x_i$'s as $x,y,z$, we have
\[
\begin{gathered}
H(x,y,z): = x^6 + y^6 + z^6 - 3x^2y^2z^2 \\ =  \tfrac 32 (x^2 y - y z^2)^2 + (x^3 - x y^2)^2 
+ \tfrac 12 (x^2 y - y^3)^2 +   (z^3 - y^2z)^2 + \tfrac 12(yz^2  - y^3)^2.
\end{gathered}
\]
For more on Hurwitz' proof, see \cite{Re2}, where Eq.~(3.5) gives a  representation of $H$
as a sum of four squares, one of which is the square of a trinomial. 

The first explicit example of a psd form which is not sos was presented in 1967 by T.  Motzkin 
\cite{M}.  It, too, arises as a substitution into the AGI: let $t_1 = x^4y^2, t_2 = x^2y^4, t_3 = 
z^6$, $\la_i = \frac 13$ and scale:
\[
M(x,y,z) := x^4 y^2 + x^2y^4 + z^6 - 3x^2y^2z^2. 
\]
The proof that $M$ is not sos was based on a preliminary argument that if $M = \sum h_j^2$,
then $h_j(x,y,z) = c_{1j}x^2y +  c_{2j}xy^2 + c_{3j}z^3 + c_{4j}xyz$: the coefficient of
$x^2y^2z^2$ in $\sum h_j^2$ is then $\sum c_{4j}^2 \neq -3$. The argument of Motzkin's
proof was formalized in \cite{Re1}, where it is shown that, in general,
 $p = \sum h_j^2$ implies that $C(h_j) \subseteq \frac 12 C(p)$.

The following machinery was developed in \cite{Re1, Re3} to analyze such forms.
Suppose
$\{ u_1, \dots, u_n \}$ with $u_i \in (2\zz_{\ge 0})^n$ and 
$\sum_{j=1}^n u_{ij} = 2d$.  We further
assume that $\uu = cvx(\{ u_1, \dots, u_n \})$ is a simplex, and that $w \in \uu \cap \zz^n$ 
has the barycentric representation $w = \sum \la_i u_i$, $\la_i \ge 0$ and $\sum_{i=1}^n \la _i 
= 1$. In this way,  the substitution $\{t_i = x^{u_i}\}$ into the AGI yields a psd form of degree 
2d,
\[
p(x) = \la_1 x^{u_1} + \cdots + \la_n x^{u_n}  - x^w.
\]
This was called an {\it agiform} in \cite{Re3}. Observe that $C(p) = \uu \cap \zz^n$.
More generally, a polynomial for which supp$(p)=\{u_1,\dots,u_n,w\}$ is called a {\it circuit  
polynomial}. Circuit polynomials have recently been studied by M. Dressler, J. Forsg\aa rd, 
S. Iliman, T. de Wolff, and J. Wang; see for example  
\cite{IdW1}, \cite{DIdW}, \cite{FdW}, \cite{W}.   Interest in circuit polynomials is in part due to 
their use in finding efficiently-computable certificates of positivity based on the AGI, which 
are then  independent of sos representations.

\vsp
There is a geometric criterion which determines whether an agiform is sos. 

\vsp
\begin{definition}   Suppose $\uu$ is given as above, and let $S \subset \mathcal U \cap \zz^n$
be  a set of lattice points containing the $u_i$'s. Then $S$ is {\it $\mathcal U$-mediated} if for
every $y \in S$, either $y = u_i$ for some $i$, or there exist $z_1 \neq z_2\in S \cap (2 \zz)^n$
 so  that $y = \tfrac12 (z_1 + z_2)$.   In other words, $S$ is $\uu$-mediated if every point in 
 $S$ is  either a vertex of $\uu$ or an average of two  different even points in $\uu$.
\end{definition}

\begin{theorem}\cite[Cor. 4.9]{Re3}\label{T:mediated}
With $\uu, \la_i$ as above, the agiform $\la_1 x^{u_1} + \cdots + \la_n x^{u_n}  - x^w$ 
is sos if and only if there is a $\uu$-mediated set containing $w$. 
 \end{theorem}
 
Up to scaling, both $H$ and $M$ are agiforms, since $w = (2,2,2)$ is the centroid 
to both $\uu_1 = \{ (4,2,0) , (2,4,0) ,(0,0,6) \}$ and 
$\uu_2 = \{(6,0,0) , (0,6,0) , (0,0,6) \}$.
By Theorem \ref{T:mediated}, M is not sos because 
$\uu_1 \cap (2\zz)^3 =\uu_1\cup\{w\}$ and it is impossible to write $w$ as an average of 
two different members of this set.  However, it is easy to check that 
the set 
\[
S = \{(6,0,0),(0,6,0),(0,0,6), (2,2,2),(4,2,0), (2,4,0),(0,2,4), (0,4,2)\}
\]
is $\uu_2$-mediated, providing an independent proof that $H$  is sos.  

We refer the reader to \cite{Re3} for the
separate proofs of the necessity and sufficiency in Theorem \ref{T:mediated}.

\section{Main Theorem}

The following theorem was asserted in \cite[Prop.2.7]{Re3}.

\begin{theorem} \label{main}   For every integer $k \geq \max\{ 2, n - 2 \}$,
$k\mathcal U \cap \zz^n$ is $(k\mathcal U)$-mediated. 
\end{theorem}

\begin{corollary}\label{blowup}
Any agiform $p \in \mathbb R[x_1,\dots,x_n]$ can be 
written as a sum of squares of forms in the variables  $x_i^{1/k}$ for $k \ge \max\{ 2, n - 2 \}$.
\end{corollary}

To prove Theorem \ref{main}, we show that $k \uu \cap \zz^n$ is $k \uu$-mediated.
That is, we show that any non-vertex  $w \in k \uu \cap \zz^n$  is the average of 
two  different  points in $k \uu \cap (2\zz)^n$.
For ease of exposition, we first prove a weaker version (Theorem \ref{case1}) 
in which $k \ge n -1$. The full proof for 
$n \ge 4$ and $k=n-2$ (Theorem \ref{full}) is more delicate. We defer the discussion of
Corollary \ref{blowup} to the next section.

We start with some notation and lemmas. First, 
recall that $t \in \rr$  may be written as  
$t = \lfloor t \rfloor + \{t\}$, where $\lfloor t \rfloor \in \zz$ and $ \{ t \}  \in [0,1)$. 
Also, if  $v = \sum a_iu_i \in k \uu$ with $a_i \in \zz_{\ge 0}, \sum a_i = k$, then
we say that $v$ is a {\it bead}. Observe that beads are always even.

\begin{lemma}  \label{lem_bead}  Suppose $k>1$ and
$v \in k \uu \cap \zz^n$ is a non-vertex bead. Then $v$ is an average of two different 
beads in $k \uu$.
\end{lemma}

\begin{proof}   Suppose that $v =  \sum a_iu_i$ is a non-vertex bead.  
At least two of the $a_i$'s must be positive; without loss of generality,
suppose $a_1,a_2 \ge 1$. Then $v$ is the average of the beads $v \pm(u_1-u_2)$ in
 $k \uu$.
\end{proof}

\begin{lemma} \label{lem_sum}

Suppose non-negative integers $b_i$ are given and $\sum_{i=1}^n b_i = R$. If $S \le R$, then
there exist non-negative integers $a_i$ so that $a_i \le b_i$ and $\sum_{i=1}^n a_i = S$.
\end{lemma}

\begin{proof}  Define the partial sums $s_k := \sum_{i=1}^k b_i$ and choose the largest 
$k$ so that
$s_k \leq S$.   Then set $a_i = b_i$ for $i = 1, \dots , k$; $a_{k+1} = S - s_k$; and 
$a_j = 0$ for $j = k+2, \dots , n$.

\end{proof}

\begin{theorem}  \label{case1} If $k \ge n - 1$, then $k \uu \cap \zz^n$ is 
$(k\mathcal U)$-mediated. 
\end{theorem}

\begin{proof}   Suppose $w \in k\mathcal U \cap \zz^n$ is not a vertex, then we must show 
that $w$ is an average of two
different  points in $k\mathcal U \cap (2\zz)^n$.    If $w$ is a bead,  we are done by 
Lemma \ref{lem_bead}, so assume
that $w$ is not a bead.   
 If we can find a bead $v = \sum a_iu_i \in k \uu$ such that $2w - v \in k\mathcal U$, 
then $w$ is the average of $v$ and $2w - v$, both of which are even.    
Further, $v \neq 2w - v$ since $v$ is a bead and $w$ is not.  

Let $w =  \sum_{i=1}^n \la_i (ku_i) = \sum_{i=1}^n \be_i u_i$; since $w$ is not a bead, at 
least one $\be_i \not \in \zz$.  
It remains to show that we can find $a_i, \dots , a_n \in \zz$ such that 
$\sum a_i u_i \in k \uu$ and
\[
2w - \sum_{i=1}^n a_i u_i = \sum_{i=1}^n (2\be_i - a_i)v_i \in k \uu.
 \]
That is, we need to show there exist $a_i \in \zz_{\ge 0}$ with $\sum a_i = k$, so that 
$2\be_i \ge a_i$ for all $i$; 
it suffices to find $a_i$ so that $\lfloor 2\be_i \rfloor \ge a_i$.

We have  $\sum_{i=1}^n \lfloor 2\be_i \rfloor  > \sum_{i=1}^n (2\be_i -1) = 2k - n$, and since 
the $ \lfloor 2\be_i \rfloor$'s and $2k-n$ are integers,  a strict inequality implies a gap
of at least 1.   Then 
\[
\sum_{i=1}^n \lfloor 2\be_i \rfloor  \ge 2k - n +1= k + (k - (n-1)) \ge k. 
\]
By  Lemma \ref{lem_sum}, this means we can find the desired $a_i$'s,
completing the proof. 
\end{proof}

The Motzkin example shows that if $n=3$, then $(3-2)\mathcal U_1 \cap \zz^3$ is not a 
mediated set; however, for larger $n$, a multiplier of $n-2$ will work. 
\begin{theorem}\label{full}   If $n \ge 4$, then $(n-2)\mathcal U \cap \zz^n$ is 
$((n-2)\mathcal U$)-mediated. 
\end{theorem}

\begin{proof}  
We shall show that if $w \in (n-2)\mathcal U \cap \zz^n$, then one of three things can
occur. In many cases, the argument of Theorem \ref{case1} can be used to write $w$ as an 
average
of a bead and another even point. If this argument fails, we can construct a ``new" 
interior point $\tu \in
\mathcal U \cap (2\zz)^n$. If $w = (n-2)\tu$, we show that $w$ is an average of two 
different even points in $(n-2)\mathcal U$. Otherwise, we may subdivide 
$\mathcal U$ into $n$ subsimplices 
$\mathcal U_{\ell}$, using $\tu$ in place of each of the vertices in turn.
Since $w$ must belong to one of the $(n-2)\mathcal U_{\ell}$'s, and is not a vertex,  we 
may repeat  the argument. 
The original simplex has only finitely many interior points, so this last case can only
be invoked finitely many times, and this will complete the proof.

Let $w = \sum_{i=1}^n \be_i u_i$ as before and assume $w$ is neither a 
vertex nor a bead. We have 
\[
\sum_{i=1}^n \lfloor 2\be_i \rfloor  \ge 1+ \sum_{i=1}^n (2\be_i -1)= 1+ 2(n-2) - n = n -3.
\]
If this sum is $\ge n-2$, then we may proceed as in the proof of Theorem \ref{case1} and find a 
bead $v$ so that $2w-v$ is in $(n-2)\mathcal U$.    

In the remaining case, 
$\sum_{i=1}^n \lfloor 2\be_i \rfloor = n-3$  implies 
\[
\sum_{i=1}^n \{2\be_i\}  = \sum_{i=1}^n (2\be_i - \lfloor 2\be_i \rfloor) =  2(n-2) -(n-3)= n-1.
 \]
Since each $\{2\be_i\}<1$, it follows that none of the summands is zero; that is, 
$2\be_i \notin \zz$.
Further,  $\sum_{i=1}^n (1-\{2\be_i\}) = n - (n-1) = 1$, and each 
$1-\{2\be_i\}$ is positive.  Define
\[
\begin{gathered}
\tilde u: = \sum_{i=1}^n (1-\{2\be_i\})u_i = \sum_{i=1}^n (1-(2\be_i -\lfloor 2\be_i \rfloor)u_i 
= \sum_{i=1}^n (1 + \lfloor 2\be_i \rfloor)u_i - 2w.
\end{gathered}
\]
Then $\tu$ is strictly interior to $\mathcal U$ (since $1- \{2\be)i\} > 0$)
and is also an even point.  In case
$w \neq (n-2)\tu$, we proceed as noted at the beginning of the proof, subdivide and
repeat. This step can only be invoked finitely many times. 

Otherwise,
\[
\begin{gathered}
w = (n-2)\tu = (n-2)\left(\sum_{i=1}^n (1 + \lfloor 2\be_i \rfloor)u_i - 2w\right) \implies \\
(2n-3)w = (n-2)\left(\sum_{i=1}^n (1 + \lfloor 2\be_i \rfloor)u_i \right) = (n-2)y,
\end{gathered}
\]
for some bead $y \in (2n-3)\mathcal U$.
Let $d_i:= 1 + \lfloor 2\be_i \rfloor \ge 0$, so that $\tilde u = \sum_{i=1}^n \frac{d_i}{2n-3} u_i$, 
where $1 \le d_i \in \zz$ and $\sum_i d_i = 2n-3$. 
Since $n \ge 4$, $2n-3 > n$, thus at least one of the $d_i$'s is $>1$.
Without loss of generality assume that $d_1 \ge 2$. 

We now note that $w = (n-2)\tilde u$ is the average of $(n-3)\tilde u + u_1$ and 
$(n-1)\tilde u - u_1$, both of which are evidently even points. 
The first is obviously in $(n-2)\mathcal U$.  The second, $(n-1)\tilde u - u_1$, can be written as
\[
(n-1) \left(\sum_{i=1}^n\frac{d_i}{2n-3} u_i \right)- u_1 = \left(\frac{(n-1)d_1}{2n-3} - 1\right)u_1 
+ \sum_{i=2}^n\frac{(n-1)d_i}{2n-3} u_i. 
\]
Since $d_1 \ge 2$, the coefficient of $u_1$ is $\ge \frac{2n-2}{2n-3} - 1 > 0$. Thus,
$(n-1)\tilde u - u_1$ is in  $(n-2)\mathcal U$, so $w$ is an average of two
different even points in $(n-2)\mathcal U$,  completing the proof. 
\end{proof}

 \section{Implication for Hilbert's 17th Problem}
 
 \begin{proof}[Proof of Corollary \ref{blowup}]
 Suppose $p(x) = \la_1 x^{u_1} + \cdots + \la_n x^{u_n}  - x^w$. Let
 \[
q(x_1,\dots,x_n) :=  p(x_1^k, \cdots, x_n^k) = \la_1 x^{ku_1} + \cdots + \la_n x^{ku_n}  - x^{kw},
 \]
which is also an agiform.  By Theorems \ref{T:mediated} and \ref{main}, $q$ is sos, and so
 \[
 q = \sum_{j=1}^r h_j^2 \implies p(x_1,\dots,x_n) =
 \sum_{j=1}^r h_j^2(x_1^{1/k},\dots,x_n^{1/k}),
 \]
 which shows that $p$ has the desired representation. 
\end{proof}
At the time that \cite{Re3} was written, and the proof given here was relegated to
the proposed preprint \cite{Re4},
the second author entertained the possibility that such a result might be true for
any psd form. Unfortunately, he discovered that the so-called ``Horn form" was a 
counterexample, and then abandoned writing   \cite{Re4}.
The Horn form was communicated to  M. Hall by A. Horn in the early 1960s,
as a counterexample 
to a  conjecture of P. H. Diananda (see \cite[p.25]{D} and \cite[p.334-5]{HN}).

Our example  comes from squaring the variables in the
Horn form, but the essence of this proof is found  in the original. Let
\[
F(x_1,\dots,x_5) = \bigg(\sum_{j=1}^5 x_j^2\bigg)^2 - 4\ \sum_{j=1}^5 x_j^2x_{j+1}^2.
\]
We view the subscripts cyclically mod 5, so that the coefficient of  $x_j^2x_k^2$ is $-2$ 
(resp. 2) if $|k-j| = 1$ (resp. $|k-j| = 2$);  $F$ is cyclically symmetric: 
\[
F(x_1,x_2,x_3,x_4,x_5) = 
F(x_2,x_3,x_4,x_5,x_1) = \cdots 
\] 
We first show that $F$ is psd. Consider $a \in \rr^5$; by the
cyclic symmetry, we may assume that $a_1^2 \le a_2^2$. We have the 
alternate representation
\[
F(x_1,\dots,x_5) = 
(x_1^2-x_2^2+x_3^2-x_4^2+x_5^2)^2 + 4(x_2^2-x_1^2)x_5^2 + 4x_1^2x_4^2, 
\]
hence  $F(a)\ge 0$, and so $F$ is psd. 

Suppose $F = \sum h_j^2$ and let the coefficient of $x_{\ell}^2$ in $h_j$
be $b_{j\ell}$. Then
\[
(x_1^2-x_2^2+x_3^2)^2 = F(x_1,x_2,x_3,0,0) = \sum_{j=1}^r h_j^2(x_1,x_2,x_3,0,0).
\]
Since the quadratic form $ h_j(x_1,x_2,x_3,0,0)$ vanishes on the (irreducible) real cone 
$g(x_1,x_2,x_3) = x_1^2-x_2^2+x_3^2=0$, it must be a multiple of $g$;
thus,  $b_{j1} = -b_{j2} = b_{j3}$. 
By cycling the variables, we
see that $b_{j2} = -b_{j3} = b_{j4}$, $b_{j3} = -b_{j4} = b_{j5}$ and $b_{j4} = -b_{j5} = b_{j1}$,
so that $b_{j1} = -b_{j1} = 0$ for all $j$. This implies that the coefficient of $x_1^4$ in $h_j$ is 
$\sum_j b_{j1}^2 = 0$, so each $h_j(x_1,x_2,x_3,0,0) = 0\cdot g$, a contradiction.

Suppose $F(x_1^k,\cdots,x_5^k)$ is sos. The proof proceeds as before,
leading to the equation
\[
(x_1^{2k}-x_2^{2k}+x_3^{2k})^2 = F(x_1^k,x_2^k,x_3^k,0,0) = \sum_{j=1}^r h_j^2(x_1,x_2,x_3,0,0).
\]
Each form $h_j(x_1,x_2,x_3,0,0)$, which has degree $2k$, vanishes on the irreducible real
variety $x_1^{2k}-x_2^{2k}+x_3^{2k}=0$, and hence must be a multiple of it. We obtain the same
fatal alternation of the coefficients of $x_\ell^{2k}$ which leads to the contradiction. Therefore, 
$F(x_1^k,x_2^k,x_3^k,x_4^k,x_5^k)$ is never sos.

\section{Implication for polytopes}

From the point of view of polytopes, one would more
naturally write $\uu = 2\mathcal P$, where $P$ is a lattice-point simplex in $\rr^n$. Further,
the conditions that the vertices lie on a hyperplane and have non-negative coefficients
seem artificial. In this way, we can drop the 
$n$-th component, so that $\mathcal P$ is the usual $n$-point lattice simplex in $\rr^{n-1}$. 

Let $d = n-1$. Then Theorem \ref{main} says that if 
$k \ge \max\{2,d-1\}$, then a non-vertex $w \in 2k\mathcal P\cap Z^d$ can be written 
as a sum of two different points in $w \in k\mathcal P\cap Z^d$. 

Requiring different points comes from the application to agiforms. There is some literature
on this subject without that requirement, which means that one needn't treat vertices as
a special case. The question then becomes: when can $w \in 2k\mathcal P\cap Z^d$ be 
written as a sum of two points in $k\mathcal P\cap Z^d$? This 
 has been studied by D. Handelman \cite{H1,H2,H3}. In particular, \cite{H3} contains a proof
 using the Shapley-Folkman Lemma that if $k \ge d-1$ (even for $n-1=d=2$), then every point
 in  $2k\mathcal P\cap Z^d$ is a sum of two (not necessarily distinct) points in 
 $k\mathcal P\cap Z^d$.

%--------------------------------------------------------------------------

 \end{document}